\documentclass[12pt]{amsart}
\usepackage{amsmath, amssymb, amsthm, latexsym}
\usepackage{amsthm}
\usepackage{fullpage}
\usepackage{amssymb}
\usepackage{latexsym}
\usepackage[center]{caption}
\usepackage{tikz}
\newcounter{braid}
\newcounter{strands}
\pgfkeyssetvalue{/tikz/braid height}{1cm}
\pgfkeyssetvalue{/tikz/braid width}{1cm}
\pgfkeyssetvalue{/tikz/braid start}{(0,0)}
\pgfkeyssetvalue{/tikz/braid colour}{black}
\pgfkeys{/tikz/strands/.code={\setcounter{strands}{#1}}}

\makeatletter
\def\cross{%
  \@ifnextchar^{\message{Got sup}\cross@sup}{\cross@sub}}

\def\cross@sup^#1_#2{\render@cross{#2}{#1}}

\def\cross@sub_#1{\@ifnextchar^{\cross@@sub{#1}}{\render@cross{#1}{1}}}

\def\cross@@sub#1^#2{\render@cross{#1}{#2}}

\def\render@cross#1#2{
  \def\strand{#1}
  \def\crossing{#2}
  \pgfmathsetmacro{\cross@y}{-\value{braid}*\braid@h}
  \pgfmathtruncatemacro{\nextstrand}{#1+1}
  \foreach \thread in {1,...,\value{strands}}
  {
    \pgfmathsetmacro{\strand@x}{\thread * \braid@w}
    \ifnum\thread=\strand
    \pgfmathsetmacro{\over@x}{\strand * \braid@w + .5*(1 - \crossing) * \braid@w}
    \pgfmathsetmacro{\under@x}{\strand * \braid@w + .5*(1 + \crossing) * \braid@w}
    \draw[braid] \pgfkeysvalueof{/tikz/braid start} +(\under@x pt,\cross@y pt) to[out=-90,in=90] +(\over@x pt,\cross@y pt -\braid@h);
    \draw[braid] \pgfkeysvalueof{/tikz/braid start} +(\over@x pt,\cross@y pt) to[out=-90,in=90] +(\under@x pt,\cross@y pt -\braid@h);
    \else
    \ifnum\thread=\nextstrand
    \else
     \draw[braid] \pgfkeysvalueof{/tikz/braid start} ++(\strand@x pt,\cross@y pt) -- ++(0,-\braid@h);
    \fi
   \fi
  }
  \stepcounter{braid}
}

\tikzset{braid/.style={double=\pgfkeysvalueof{/tikz/braid colour},double distance=1pt,line width=2pt,white}}

\newcommand{\braid}[2][]{%
  \begingroup
  \pgfkeys{/tikz/strands=2}
  \tikzset{#1}
  \pgfkeysgetvalue{/tikz/braid width}{\braid@w}
  \pgfkeysgetvalue{/tikz/braid height}{\braid@h}
  \setcounter{braid}{0}
  \let\sigma=\cross
  #2
  \endgroup
}
\makeatother

\input xypic
\newtheorem{theorem}{Theorem}
\newtheorem{proposition}[theorem]{Proposition}

\newtheorem{lemma}[theorem]{Lemma}
\newtheorem{conjecture}[theorem]{Conjecture}
\newtheorem{corollary}[theorem]{Corollary}
\theoremstyle{definition}

\newtheorem{definition}[theorem]{Definition}
\newtheorem{construction}[theorem]{Construction}
\newtheorem{remark}[theorem]{Remark}


\def\Z{\mathbb{Z}}

\def\Pi{\mathbb{P}^{\infty}}

\def\qed{\hfill$\square$\medskip}

\def\Zpk{\mathbb{Z}/p^{k}}
\def\Zpk1{\mathbb{Z}/p^{k-1}}

\def\sl2{\widetilde{SL_{2}(\Z)}}

\DeclareMathOperator{\ord}{ord}
\DeclareMathOperator{\Exp}{Exp}

\title{On the Maximal Cross Number of Unique Factorization Indexed Multisets}
\author{Daniel Kriz\\ Mathematics Department, Princeton University}
\email{dkriz@princeton.edu}

\begin{document}

\begin{abstract} In this paper, we study a conjecture of Gao and Wang concerning a proposed formula $K_1^*(G)$ for the maximal cross number $K_1(G)$ taken over all unique factorization indexed multisets over a given finite abelian group $G$. As a corollary of our first main result, we verify the conjecture for abelian groups of the form $C_{p^m}\oplus C_p, C_{p^m}\oplus C_q, C_{p^m}\oplus C_q^2$, $C_{p^m}\oplus C_r^n$  where $p,q$ are distinct primes and $r\in\{2,3\}$. In our second main result we verify that $K_1(G) = K_1^*(G)$ for groups of the form $C_r\oplus C_{p^m}\oplus C_p, C_{rp^mq}$ and $C_r\oplus C_p \oplus C_q^2$ for $r \in \{2,3\}$ given some restrictions on $p$ and $q$. We also study general techniques for computing and bounding $K_1(G)$, and derive an asymptotic result which shows that $K_1(G)$ becomes arbitrarily close to $K_1^*(G)$ as the smallest prime dividing $|G|$ goes to infinity, given certain conditions on the structure of $G$. We also derive some necessary properties of the structure of unique factorization indexed multisets which would hypothetically violate $k(S) \le K_1^*(G)$. 
\end{abstract}

\maketitle

\section{Introduction and Preliminaries}\label{introduction}
Throughout let $(G,+)$ be a finite abelian group (written additively). 
Let $S = \{g_1,\ldots,g_{\ell}\}$, $\ell\in \mathbb{N}$, be a finite indexed multiset of elements of $G$. In \cite{GaoWang}, Gao and Wang consider sequences rather than indexed multisets. However, the notion of indexed multisets seems more natural in the context of our discussion, as giving an ordering on the elements of our set is unnecessary and we only need the indexing to distinguish between copies of the same element. To any subset $I\subseteq [\ell]$, we associate a submultiset $S(I) := \{g_i\in S : i\in I\} \subseteq S$. Let 
$$\sigma(S) := \sum_{g\in S} g$$
denote the sum of the elements of $S$ (with multiplicity). By convention $\sigma(\emptyset) = 0$. We call $S$ \emph{zero-sum} if $\sigma(S) = 0$, we call $S$ \emph{minimal zero-sum} if $\sigma(S) = 0$ and for any $\emptyset\subsetneq S'\subsetneq S$ we have $\sigma(S') \neq 0$ and we call $S$ \emph{zero-sum free} if for any $\emptyset \subsetneq S' \subseteq S$ we have $\sigma(S')\neq 0$.  For any indexed multiset $S$ over $G$, let $|S|$ denote the number of elements of $S$ counted with multiplicity. 

Now for any indexed multiset $S = \{g_1,\ldots,g_{\ell}\}$ over $G\setminus\{0\}$ (i.e., with elements contained in $G\setminus\{0\}$), an \emph{irreducible factorization} of $S$ is a decomposition of the indexing set $[\ell]$
$$[\ell] = \bigsqcup_{i=1}^m I_i$$
where $S(I_i)$ is minimal zero-sum for each $1\le i\le m$. We often refer to the $S(I_i)$ as \emph{irreducible factors} of the irreducible factorization $\bigsqcup_{i=1}^m I_i$. We consider two irreducible factorizations $\bigsqcup_{i=1}^m I_i$ and $\bigsqcup_{j=1}^n J_j$ equivalent if and only if $m=n$ and $\{I_1,\ldots,I_m\} = \{J_1,\ldots,J_n\}$. A zero-sum indexed multiset $S$ over $G\setminus\{0\}$ with precisely one equivalence class of irreducible factorizations is called a \emph{unique factorization indexed multiset} (which we will henceforth denote by ``UFIM" for brevity).

The above notions have interpretations in algebraic number theory, see \cite{BaginskiChapman}.

For an element $g\in G$, let $\ord(g)$ denote its order in $G$, i.e., the smallest positive integer $n$ such that $ng = 0$. Now let $G = \bigoplus_{i=1}^r C_{n_i}$ be the unique decomposition of $G$ into a direct sum of cyclic groups such that $n_i | n_{i+1}$ for each $1\le i\le r$, and $n_1>1$. We call $r$ the \emph{rank} of $G$ and $n_r = \Exp(G)$ the \emph{exponent} of $G$. We now define the \emph{cross number}, the main quantity we will be studying.
\begin{definition}[Cross Number]
For any indexed multiset $S$ over $G$, we define the \emph{cross number of $S$} by 
$$k(S) :=\sum_{g\in S} \frac{1}{\ord(g)}$$ 
(by convention $k(\emptyset) = 0$) and define
$$K_1(G) := \max\{k(S) : S \;\text{is a UFIM over} \;G\setminus\{0\}\}.$$
\end{definition}
For a finite abelian group $G$, decompose $G$ into the direct sum of prime-power cyclic groups: $G = \bigoplus_{i=1}^n \bigoplus_{j=1}^{n_i} C_{p_i^{e_{ij}}}$ where the $p_i$ are distinct primes. Put
$$K_1^*(G) := \sum_{i=1}^n \sum_{j=1}^{n_i} \frac{p_i^{e_{ij}} - 1}{p_i^{e_{ij}}-p_i^{e_{ij}-1}} = \sum_{i=1}^n \sum_{j=1}^{n_i} \sum_{k=1}^{e_{ij}} \frac{1}{p_i^{k-1}}.$$
Again by convention $K_1^*((\{0\},+)) = 0$. Note that for any finite abelian groups $G$ and $H$, we have $K_1^*(G\oplus H) = K_1^*(G) + K_1^*(H)$. 

Gao and Wang gave the following conjecture, which will be the main focus of this paper. 

\begin{conjecture}[Gao-Wang \cite{GaoWang}] \label{K1conjecture} For any finite abelian group $G$, we have the equality 
$$K_1(G) = K_1^*(G).$$
\end{conjecture}

Note that Conjecture \ref{K1conjecture} is equivalent to the statement that both $K_1(C_{p^m}) = \frac{p^m - 1}{p^m-p^{m-1}}$ for any prime-power cyclic group $C_{p^m}$ and $K_1$ is additive over direct sums, i.e. for any two finite abelian groups $G,H$ we have $K_1(G\oplus H) = K_1(G) + K_1(H)$. 

In \cite{GaoWang}, Gao and Wang show that Conjecture \ref{K1conjecture} partially holds.

\begin{proposition}[\cite{GaoWang}]
\label{lowerbound}
For any finite abelian group $G$, we have 
$$K_1(G) \ge K_1^*(G).$$
\end{proposition}

\begin{remark} \label{structureremark} In \cite{GaoWang}, Gao and Wang construct the following UFIM whose cross number equals $K_1^*(C_{p^m})$, in particular proving Proposition \ref{lowerbound}. For any $x\in C_{p^m}\setminus\{0\}$, let $S_x^k$ denote the indexed multiset over $C_{p^m}\setminus\{0\}$ in which $x$ occurs $k$ times. Let $\gamma$ be a generator of $C_{p^m}$ and take the indexed multiset
$$S = \left(\bigsqcup_{i=1}^m  S_{p^{i-1}\gamma}^{p-1}\right) \sqcup \left(\bigsqcup_{i=1}^m S_{(1-p)p^{i-1}\gamma}^1\right).$$
\end{remark}

Gao and Wang verified Conjecture \ref{K1conjecture} in \cite{GaoWang} for special families of abelian groups.

\begin{theorem}[Gao-Wang \cite{GaoWang}] \label{GaoWang} Conjecture \ref{K1conjecture} holds, i.e., $K_1(G) = K_1^*(G)$, for $G$ of the form:
\begin{enumerate}
\item $C_{p^m}$, $p$ prime, $m\in\mathbb{N},$
\item $C_{pq}$, $p, q$ prime$,$
\item $C_2^m$, $m\in\mathbb{N},$
\item $C_3^m$, $m\in\mathbb{N},$
\item $C_p^2$, $p$ prime.
\end{enumerate} 
\end{theorem}
The first main result of this paper, proven in Section \ref{mainresult}, is

\begin{theorem} \label{mainthm} Let $p,q$ be distinct primes and $m,n\in \mathbb{N}$. Then
\begin{enumerate}
\item\label{1} $K_1(C_{p^m}\oplus C_p^n) \le K_1(C_{p^m}) + K_1(C_p^{n+1})-1,$
\item \label{2} $K_1(C_{p^m}\oplus C_q^n) \le K_1(C_{p^m}) + K_1(C_q^n).$
\end{enumerate}
\end{theorem}

This result in particular verifies Conjecture \ref{K1conjecture} for more families of abelian groups:

\begin{corollary} \label{corollary} For $p,q$ distinct (possibly even) primes, and any $m,n\in\mathbb{N}$, we have $K_1(G) = K_1^*(G)$ for the following groups $G$:
\begin{enumerate}
\item $C_{p^m}\oplus C_p,$
\item $C_{p^m}\oplus C_q,$
\item $C_{p^m}\oplus C_q^2,$
\item $C_{p^m} \oplus C_2^n,$
\item $C_{p^m}\oplus C_3^n.$
\end{enumerate}
\end{corollary}

Our second main result, proven in Section \ref{2pq}, concerns the families for which Conjecture \ref{K1conjecture} ``eventually" holds.

\begin{theorem} \label{mainthm2} Fix any $c\in\mathbb{R}_{\ge1}$ and $r\in \{2,3\}$. Suppose $G = \bigoplus_{i=1}^n\bigoplus_{j=1}^{n_i} C_{p_i^{e_{ij}}}$, where  $p_i>r$ are distinct primes for $1\le i\le n$, and  $p_1< \cdots < p_n< cp_1$ if $n>1$, is a finite abelian group with $K_1(G) = K_1^*(G)$ and $k(C_r\oplus G) = k^*(C_r\oplus G)$. Then if $p_1$ is large enough so that 
$$\frac{1}{r} + \frac{1}{p_1}K_1^*\left(\bigoplus_{j=1}^{n_1} C_{p_1^{e_{1j}}}\right) + \sum_{i=2}^n \sum_{j=1}^{n_i} \frac{(cp_1)^{e_{ij}}-1}{(cp_1)^{e_{ij}+1} - (cp_1)^{e_{ij}}} \ge \frac{\log_2 (rc^{(\sum_{i=2}^{n} \sum_{j=1}^{n_i} e_{ij})} p_1^{(\sum_{i=1}^n \sum_{j=1}^{n_i} e_{ij})})}{p_1}$$
(note that as $p_1\rightarrow \infty$, the left hand side tends to $\frac{1}{r}$ while the right hand side tends to 0), we have 
$$K_1(C_r\oplus G) = K_1^*(C_r\oplus G).$$

Furthermore, if equality does not hold in the constraint for $p_1$ above, then any UFIM $S$ over $(C_r\oplus G)\setminus\{0\}$ with $k(S) = K_1(C_r\oplus G)$ has a decomposition 
$$S = S_r \sqcup S_G$$
such that $S_r$ is a UFIM over $C_r\setminus\{0\}$ and $S_G$ is a UFIM over $G\setminus\{0\}$. 
\end{theorem}

\begin{corollary} \label{secondmaincorollary} Let $c\in\mathbb{R}_{> 1},\; m\in\mathbb{N},  \;r\in\{2,3\}$ where $r<p<q$ are distinct primes with $q\le cp$. 
\begin{enumerate}
\item If $G = C_r \oplus C_{p^m}\oplus C_p$, we have that $K_1(C_r\oplus C_{p^m} \oplus C_p) = K_1^*(C_r\oplus C_{p^m} \oplus C_p)$ for all $p$ large enough so that 
$$\frac{1}{r} + \frac{p^m-1}{p^{m+1} - p^m} + \frac{1}{p}\ge \frac{\log_2 (rp^{m+1})}{p}.$$
\item If $G = C_{rp^mq}$, we have that $K_1(C_{rp^mq}) = K_1^*(C_{rp^mq})$ for all $p$ large enough so that 
$$\frac{1}{r} + \frac{p^m-1}{p^{m+1}-p^m}  + \frac{1}{cp} \ge \frac{\log_2 (rcp^{m+1})}{p}.$$
\item If $G = C_{rpq^m}$, we have that $K_1(C_{rpq^m}) = K_1^*(C_{rpq^m})$ for all $p$ large enough so that 
$$\frac{1}{r} + \frac{1}{p} + \frac{(cp)^m-1}{(cp)^{m+1}-(cp)^m}\ge \frac{\log_2 (rc^mp^{m+1})}{p}.$$ 
\item If $G = C_r\oplus C_{p^m} \oplus C_q^2$, we have that $K_1(C_r\oplus C_{p^m}\oplus C_q^2) = K_1^*(C_r\oplus C_{p^m} \oplus C_q^2)$ for all $p$ large enough so that 
$$\frac{1}{r} + \frac{p^m-1}{p^{m+1}-p^m} + \frac{2}{cp} \ge \frac{\log (rc^2p^{m+2})}{p}.$$
\item If $G = C_r\oplus C_p^2 \oplus C_{q^m}$, we have that $K_1(C_r\oplus C_p^2 \oplus C_{q^m}) = K_1^*(C_r\oplus C_p^2 \oplus C_{q^m})$ for all $p$ large enough so that 
$$\frac{1}{r} + \frac{2}{p} + \frac{(cp)^m-1}{(cp)^{m+1}-(cp)^m} \ge \frac{\log_2 (rc^mp^{m+2})}{p}.$$
\end{enumerate}

Moreover for each of the families $C_r\oplus G$ above, if equality in the corresponding constraint for $p$ does not hold, then any UFIM $S$ over $(C_r\oplus G)\setminus\{0\}$ with $k(S) = K_1(G)$ has a decomposition 
$$S = S_r\sqcup S_G$$
such that $S_r$ is a UFIM over $S_r\setminus\{0\}$ and $S_G$ is a UFIM over $G\setminus\{0\}$. 
\end{corollary}

This paper is organized as follows. In Section \ref{invariants}, we give a brief survey of other zero-sum group invariants. We will utilize these invariants in the methods used to prove our main results in Sections \ref{mainresult} and \ref{2pq}. In Section \ref{outline}, we give a brief outline of our general method for proving $K_1(G) = K_1^*(G)$ and bounding $K_1(G)$. In Section \ref{lemmas}, we prove several fundamental lemmas which will be used throughout our paper. In Section \ref{mainresult}, we prove the first main results of this paper, Theorem \ref{mainthm} and Corollary \ref{corollary}. In Section \ref{structuralresults}, we study some properties of general UFIMs and derive some key results which will be used in the proof of Theorem \ref{mainthm2}.  In Section \ref{2pq} we prove our second main result, Theorem \ref{mainthm2} and Corollary \ref{secondmaincorollary},  calculating $K_1(G)$ for certain subsets of the families $C_r\oplus C_{p^m} \oplus C_p,\; C_{rp^mq}$ and $C_r \oplus C_p^2 \oplus C_q$, showing that Conjecture \ref{K1conjecture} ``eventually" holds for members of this subsets.

 In Section \ref{secbound}, we study the asymptotic behavior of $K_1(G)$, in particular showing that it behaves essentially like $k(G)$ and $K_1^*(G)$, and that it becomes arbitrarily close to these quantities in a certain limit. This gives new information on the  behavior of $K_1(G)$. We also give an even sharper bound on $K_1(G) - K_1^*(G)$ in the case of certain classes of finite abelian groups, including finite abelian $p$-groups $G$. 

\section{A Brief Survey of Related Group Invariants} \label{invariants}
Group invariants such as the cross number have proven useful in the study of factorization problems in Krull domains (see \cite{Chapman}), and in the study of block monoids (see \cite{KrauseZahlten}). In this section, we recall other invariants related to zero-sum indexed multisets over finite abelian groups. We include this brief survey of known results both to serve as a reference for the reader and because these quantities will appear in the methods we use to study $K_1(G)$ throughout the rest of the paper. A reader already familiar with the material below may safely skip this section.

The following invariants quantify the maximal length of certain types of zero-sum indexed multisets over $G\setminus\{0\}$:
$$D(G) := \max\{|S| : S \;\text{is a minimal zero-sum indexed multiset over} \;G\setminus\{0\}\}$$
$$N_1(G) := \max\{|S| : S \;\text{is a UFIM over} \;G\setminus\{0\}\}.$$
We refer to $D(G)$ as the \emph{Davenport constant} of $G$ and $N_1(G)$ as the \emph{first Narkiewicz constant} (or simply the \emph{Narkiewicz constant} of $G$), introduced by Narkiewicz in \cite{Narkiewicz}. Similarly to $K_1(G)$, the Narkiewicz constant $N_1(G)$ has a conjectured explicit formula.

\begin{conjecture}[Narkiewicz \cite{NarkiewiczSliwa}] \label{N1conjecture} For a given abelian group $G$, write it as a sum of invariant factors: $G = \bigoplus_{i=1}^r C_{n_i}$ where $n_i|n_j$ if $i\le j,\; n_1>1$. Then 
$$N_1(G) = \sum_{i=1}^r n_i.$$
\end{conjecture}

A resolution of Conjecture \ref{N1conjecture} still seems far away, but it has been verified for the following special cases.

\begin{theorem}[\cite{Gao}, \cite{GaoLiPeng}, \cite{NarkiewiczSliwa}] \label{N1} Conjecture \ref{N1conjecture} holds for:
\begin{enumerate}
\item $C_n$ where $n\in \mathbb{N}$;
\item $C_2^m$ where $m\in\mathbb{N}$;
\item $C_3^m$ where $m\in\mathbb{N}$;
\item $C_p^2$ where $p$ is prime.
\end{enumerate}
\end{theorem}

The Davenport constant $D(G)$ has a similar associated formula 
$$D^*(G) = D^*\left(\bigoplus_{i=1}^rC_{n_i}\right) := 1 + \sum_{j=1}^{r} (n_i -1).$$
$D(G)$ and $D^*(G)$ are known to be equal for groups of rank at most 2, but have been shown to differ in certain groups of rank at least 4; they are conjectured to be equal for groups of rank 3 (see \cite{GeroldingerLiebmannPhilipp}). 

We also have an invariant similar to $K_1(G)$, by instead taking the maximal cross number over minimal zero-sum indexed multisets:
$$K(G) := \max\{k(S) : S \;\text{is a minimal zero-sum indexed multiset over} \;G\setminus\{0\}\}.$$
The invariant $K(G)$, often simply called the \emph{cross number} of $G$, was introduced by Krause in \cite{Krause} (for further information, see \cite{GaoGeroldinger}, \cite{GeroldingerHalter-Koch}, \cite{Geroldinger}, \cite{GeroldingerSchneider}, \cite{GeroldingerSchneider2}, and \cite{Girard}). Like $D(G)$ and $N_1(G)$, $K(G)$ has only been fully computed for some families of finite abelian groups, including $p$-groups. We have the following conjecture.

\begin{conjecture}[Krause-Zahlten \cite{KrauseZahlten}] \label{Kconjecture} For any finite abelian group $\bigoplus_{i=1}^r\bigoplus_{j=1}^{t_i} C_{p_i^{e_{ij}}}$, we have 
$$K(G) = K^*(G) := \frac{1}{\Exp(G)} + \sum_{i=1}^r\sum_{j=1}^{t_i} \frac{p_i^{e_{ij}}-1}{p_i^{e_{ij}}}.$$
\end{conjecture}

Conjecture \ref{Kconjecture} has been verified for some families, given by the following Theorem.

\begin{theorem}
\label{Kbound}
Conjecture \ref{Kconjecture} holds for the following families of abelian groups $G$:
\begin{enumerate}
\item (See \cite{Geroldinger}) Finite abelian p-groups for any prime $p$. 
\item (See \cite{GeroldingerSchneider}) Groups of the form $C_{p^m}\oplus C_{p^n} \oplus C_q^s$ for distinct primes $p,q$ and $m,n,s\in\mathbb{N}$.
\item (See \cite{GeroldingerSchneider}) Groups of the form $\bigoplus_{i=1}^n C_{p_i^{e_i}} \oplus C_q^s$ where $p_1,\ldots,p_n,q$ are distinct primes, $m,n\in\mathbb{N}$, $s\in\mathbb{N}\cup \{0\}$, and one of the following conditions holds:
\begin{enumerate}
\item $n\le 3$ and $p_1\cdots p_n \neq 30$.
\item $p_k\ge k^3$ for every $1\le k\le n$. 
\end{enumerate}
\item  (See \cite{KrauseZahlten}) Cyclic groups of the form $G = C_{p^mq}$ where $p,q$ are distinct primes and $m\in \mathbb{N}$.
\item (See \cite{KrauseZahlten}) Cyclic groups of the form $G = C_{p^2q^2}$ where $p,q$ are distinct primes.
\item (See \cite{KrauseZahlten}) Cyclic groups of the form $G = C_{pqr}$ where $p,q,r$ are distinct primes.
\end{enumerate}
\end{theorem}

We can also define the \emph{little cross number} of $G$:
$$k(G) := \max\{k(S) : S \;\text{is a zero-sum free indexed multiset over}\; G\setminus\{0\}\}.$$
\begin{remark}Note that any zero-sum free indexed multiset differs by one element from some minimal zero-sum indexed multiset: If $S$ is zero-sum free, then $S \sqcup \{-\sigma(S)\}$ is minimal zero-sum. In particular, for any zero-sum free $S$, we have $k(S) + \frac{1}{\Exp(G)} \le k(S\sqcup \{-\sigma(S)\})$, and so we have the following Proposition.
\end{remark}
\begin{proposition}\label{zero-sumfree} For any finite abelian group $G = \bigoplus_{i=1}^r\bigoplus_{j=1}^{t_i} C_{p_i^{e_{ij}}}$, we have 
$$k(G) + \frac{1}{\Exp(G)} \le K(G).$$
\end{proposition} 
We again have a conjectured explicit formula for $k(G)$.
\begin{conjecture}\label{kconjecture} For any finite abelian group $G = \bigoplus_{i=1}^r\bigoplus_{j=1}^{t_i} C_{p_i^{e_{ij}}}$ written as a direct sum of prime-power cyclic groups, we have
$$k(G) = k^*(G) :=  \sum_{i=1}^r\sum_{j=1}^{t_i} \frac{p_i^{e_{ij}}-1}{p_i^{e_{ij}}}.$$
\end{conjecture}

\begin{remark}Again, for any given abelian group $G$ one can construct a zero-sum free indexed multiset $S$ such that $k(S) = k^*(G)$, and hence we have $k(G) \ge k^*(G)$.
\end{remark}
\begin{remark} \label{kremark} Note that given a finite abelian group $G$ for which we have $K(G) = K^*(G)$, then Proposition \ref{zero-sumfree} along with $k(G) \ge k^*(G)$ implies that $k(G) = k^*(G)$. Hence Conjecture \ref{kconjecture} holds for the families of abelian groups given in Proposition \ref{Kbound}.
\end{remark}

We include the following table summarizing the main families of abelian groups for which $D(G), N_1(G), K(G), k(G)$ and $K_1(G)$ have been fully computed (including the results shown in this paper). The values of the invariants for the familes listed below are all consistent with their corresponding conjectures. 
\begin{center}
    \begin{tabular}{ | l | p{12cm}| }
    \hline
    Invariant & Fully computed for \\ \hline
    $D(G)$ & $p$-groups where $p$ is a prime, cyclic groups, $C_m\oplus C_n$ where $m,n\in\mathbb{N}, m|n$ \\ \hline
    $N_1(G)$ & cyclic groups, $C_2^m$, $C_3^m$, $C_p^2$ where $p$ is prime, $m\in \mathbb{N}$\\ \hline
    $K(G)$ & $p$-groups where $p$ is a prime, $C_{p^m}\oplus C_{p^n} \oplus C_q^s$, $\bigoplus_{i=1}^n C_{p_i^{e_i}} \oplus C_q^s$ where $p_1,\ldots,p_n$ are distinct primes satisfying certain conditions and $n\in\mathbb{N}$, $s\in\mathbb{Z}_{\ge0}$, $C_{p^mq}$, $C_{p^2q^2}$, $C_{pqr}$ where $p, q,r$ are distinct primes \\ \hline
    $k(G)$ & same as $K(G)$ \\ \hline
    $K_1(G)$ & $C_{p^m}$, $C_{p^m}\oplus C_p$, $C_{p^mq}$, $C_{p^m}\oplus C_q^2$, $C_{p^m} \oplus C_r^n$, $C_r\oplus C_{p^m}\oplus C_p, C_{rp^mq}$ and $C_r\oplus C_p^2\oplus C_q$ (under certain conditions), where $p,q$ are distinct (possibly even) primes, $m,n \in\mathbb{N}$, and $r\in\{2,3\}$\\
    \hline
    \end{tabular}
\end{center}

\section{An Outline of our Method}\label{outline}
For a given finite abelian group $G$, our general stategy will be to find a bound on $K_1(G)$ of the form
$$K_1(G) \le K_1^*(G) + [\text{extra terms}].$$
To do this, we choose a suitable subgroup $H \le G$ and using the quotient map $G\rightarrow G/H$ derive a bound of the form
$$K_1(G) \le K_1(H) + K_1(G/H) + [\text{extra terms}].$$
If we have $G\cong H\oplus G/H$ and it is known that $K_1(H)  = K_1^*(H)$ and $K_1(G/H) = K_1^*(G/H)$, so that $K_1^*(H) + K_1^*(G/H) = K_1^*(H\oplus G/H) = K_1^*(G)$, then the above inequality becomes
$$K_1(G) \le K_1^*(G) + [\text{extra terms}].$$
Ideally we would hope to simply get $K_1(G) \le K_1^*(G)$ in this way, which by Proposition \ref{lowerbound} would imply that $K_1(G) = K_1^*(G)$, but in most cases it seems that we can only show that the ``extra terms" are small. In Section \ref{lemmas}, we derive a general bound
$$K_1(G) \le K_1(G/H) + N_1(H)\cdot K(G/H)$$
but often we wish to obtain a better bound than this. To do so, we will often need to treat each case using \emph{ad hoc} methods, as in the proof of our first main result in Section \ref{mainresult}. 

\section{Lemmas} \label{lemmas}
In this section, we develop techniques, inspired by the arguments of Gao and Wang in \cite{GaoWang}, which will be used throughout the remainder of this paper. 

We first make the following observation.

\begin{remark} \label{factor} Given a UFIM $S$ over $G\setminus\{0\}$, for any $S'\subseteq S$ with $\sigma(S') = 0$, we have that $S'$ is a union of irreducible factors of $S$, and hence must have a unique factorization. Hence $S'$ is also a UFIM.
\end{remark}

We have the following useful reformulation of the notion of unique factorization.

\begin{proposition}[Equivalent Characterization of Unique Factorization, see \cite{Narkiewicz}]
\label{equivchar}
Let $G$ be a finite abelian group, and let $S$ be a zero-sum indexed multiset over $G\setminus\{0\}$. Then the following two conditions are equivalent:

\begin{enumerate}
\item \label{uniquefactorization} $S$ is a UFIM;
\item \label{refinement} For any two zero-sum submultisets $S_1$ and $S_2$ of $S$, the intersection $S_1\cap S_2$ is also a zero-sum indexed multiset.
\end{enumerate}
\end{proposition}

\begin{proposition} \label{lift}
Given a UFIM $S$ over $G\setminus\{0\}$, for any submultiset $S_0 \subset S$ with $\sigma(S_0)\neq 0$, $(S\setminus S_0) \sqcup \{\sigma(S_0)\}$ is also a UFIM over $G\setminus\{0\}$. 
\end{proposition}

\begin{proof} We have a map from irreducible factorizations of $(S\setminus S_0) \sqcup \{\sigma(S_0)\}$  to irreducible factorizations of $S$ given by deleting the irreducible factor $T$ containing $\{\sigma(S_0)\}$ and replacing it with an irreducible factorization of ($T\setminus\{\sigma(S_0)\})\sqcup S_0$.  This map has a left inverse given by replacing the smallest union $U$ of irreducible factors of $S$ containing $S_0$ with an irreducible factorization of $(U\setminus S_0)\sqcup \{\sigma(S_0)\}$. Hence the original map is injective. Thus since $S$ is a UFIM, so is $(S\setminus S_0) \sqcup \{\sigma(S_0)\}$. 
\end{proof}
Note that any map of groups $\phi: G\rightarrow G'$ induces an action on indexed multisets given by $\phi(S) = \{\phi(g_1),\ldots,\phi(g_{\ell})\}$ for $S = \{g_1,\ldots,g_{\ell}\}$. Observe that $\phi(S)$ is zero-sum if and only if $\sigma(S) \in \text{ker}(\phi)$.

\begin{remark} If $S$ is zero-sum, then $\phi(S)$ is zero-sum, but if $S$ is a UFIM, $\phi(S)$ is not necessarily a UFIM. For example, consider the UFIM over $C_3^2\setminus\{0\} = \mathbb{Z}_3^2\setminus\{0\}$:
$$S = \{(1,1), (2,2), (1,2), (2,1)\} = \{(1,1), (2,2)\}\sqcup \{(1,2), (2,1)\}$$ 
and the projection onto the first factor $\phi: C_3^2 \rightarrow C_3$. 
\end{remark}

We will use the following construction for the rest of our discussion.
\begin{construction} \label{construction} Given a group $G$, suppose we have a surjective group homomorphism $\phi : G \rightarrow G'$ and a UFIM $S$ over $G\setminus\{0\}$. Let $T(\phi)=\{x\in S : x\in \ker(\phi)\}$ (when the choice of $\phi$ is clear, we will simply write $T(\phi) = T$), let $S' = S\setminus T(\phi)$, and let $t\in \mathbb{Z}_{\ge0}$ be maximal such that there exist disjoint zero-sum free submultisets $S_1,\ldots,S_t$ of $S'$ such that for each $1\le i\le t$, $\sigma(S_i) \in \ker(\phi)\setminus\{0\}$. (Note we are slightly abusing notation: when $t=0$, there exists no $S_0\subseteq S'$ such that $S_0$ is zero-sum free and $\sigma(S_0) \in \ker(\phi)\setminus\{0\}$. For further interpretation of the case $t=0$, see Remark \ref{crossterms} below.) Let $S'' = S'\setminus\left(\bigsqcup_{i=1}^t S_i\right)$. Note that  for each $1\le i \le t$, $\phi(S_i)$ is a minimal zero-sum indexed multiset over $G'\setminus \{0\}$, seen as follows: for any $U_i \subsetneq S_i$ with $\phi(U_i)$ zero-sum, then since $S_i$ is zero-sum free, $\sigma(U_i) \in \ker(\phi)\setminus\{0\}$, which contradicts the maximality of $t$. 

Now we have
$$S = T(\phi) \sqcup S'' \sqcup \bigsqcup_{i=1}^t S_i$$
which implies
$$k(S) = k(T(\phi)) +  k(S'') + \sum_{i=1}^t k(S_i).$$
We will seek to bound $k(S)$ (and ultimately $K_1(G)$) by bounding each of the three summands on the right hand side.
\end{construction}

\begin{remark} \label{crossterms} When $\ker(\phi)$ is a direct factor of $G$, the zero-sum free submultisets $S_i$ in the construction above represent ``cross terms" in $S$, i.e., elements which do not belong to a single direct factor of $G \cong \ker(\phi) \oplus (G/\ker(\phi))$. As $\sigma(S_i) \in \ker(\phi)\setminus\{0\}$, the elements of $S_i$ can be thought of as adding together to ``cancel out" their $G/\ker(\phi)$ components . In particular, for any submultiset $S$ with $t=0$, each element of $S$ belongs to either $\ker(\phi)$ or $G/\ker(\phi)$. 
\end{remark}

\begin{remark}\label{phibound} Note now that for any group homomorphism $\phi: G\rightarrow G'$ and any $x\in G$, $\ord(\phi(x))\le \ord(x)$, and so for any UFIM $S$ over $G\setminus\{0\}$, we have $k(S') \le k(\phi(S'))$. 
\end{remark}

\begin{proposition} 
\label{phiunique} \label{split} \label{phibound2} \label{bound} \label{generalbound} \label{push}
In the notation of Construction \ref{construction}, we have that $S''$ is a UFIM over $G\setminus\{0\}$, $\phi(S'')$ is a UFIM over $(G/\ker(\phi))\setminus\{0\}$ and $T\sqcup \bigsqcup \{\sigma(S_i)\}$ is a UFIM over $\ker(\phi)\setminus\{0\}$. 
\label{remark}
As a consequence, we have the following:
\begin{enumerate}
\item \label{phiunique1}$k(T) + \sum_{i=1}^t k(\{\sigma(S_i)\}) = k(T\sqcup\bigsqcup_{i=1}^t \{\sigma(S_i)\}) \le K_1(\ker(\phi))$;
\item \label{phiunique2}$|T|+t =|T\sqcup\bigsqcup_{i=1}^t \{\sigma(S_i)\}| \le N_1(\ker(\phi))$;
\item \label{split3} $k(S) \le K_1(\ker(\phi)) + k(S')$;
\item \label{phibound2,4} $k(S) \le K_1(\ker(\phi)) + k(\phi(S'))$;
\item \label{bound5} $k(\phi(S')) \le K_1(G/\ker(\phi)) + t\cdot K(G/\ker(\phi))$;
\item \label{generalbound6} $K_1(G) \le K_1(G/\ker(\phi)) + N_1(\ker(\phi))\cdot K(G/\ker(\phi))$;
\item \label{push7}  if $t=0$, then $k(S) \le K_1(\ker(\phi)) + K_1(G/\ker(\phi))$.
\end{enumerate}
\end{proposition}

\begin{proof}
We first show that $\phi(S'')$ is a UFIM. Since 
$$\sigma(\phi(S\setminus S'')) = \sigma\left(\phi\left(T\sqcup \bigsqcup_{i=1}^t S_i\right)\right) = \sigma(\phi(T)) + \sum_{i=1}^t\phi(\sigma(S_i)) = 0,$$
we have $\sigma(\phi(S''))=0$. Now choose any two zero-sum submultisets $\phi(U_1)$ and $\phi(U_2)$ of $\phi(S'')$. Then $\sigma(U_1),\sigma(U_2)\in\text{ker}(\phi)$. By maximality of $t$, we must have $\sigma(U_1)= \sigma(U_2) = 0$. Now since $S$ is a UFIM, by Proposition \ref{equivchar}, we have $\sigma(U_1\cap U_2) = 0$, and hence $\sigma(\phi(U_1\cap U_2)) = \sigma(\phi(U_1) \cap\phi(U_2)) = 0$. Since $\phi(U_1)$ and $\phi(U_2)$ were arbitrary zero-sum submultisets of $\phi(S'')$, again by Proposition \ref{equivchar}, we have that $\phi(S'')$ is a UFIM. 

We now show that $S''$ is a UFIM. Suppose now that $S''$ is not zero-sum, so that we can choose a zero-sum free submultiset $U\subseteq S''$. Then $\sigma(U)\in \ker(\phi)\setminus\{0\}$ which contradicts the maximality of $t$ (see Construction \ref{construction}). Hence $S''$ is zero-sum. Thus, since $S$ is a UFIM over $G\setminus\{0\}$ and $S'' \subseteq S$ with $\sigma(S'') = 0$, then $S''$ is also a UFIM over $G\setminus\{0\}$ (see Remark \ref{factor}). Now since $S''$ is a UFIM, so is $S\setminus S'' = T\sqcup \bigsqcup_{i=1}^t S_i'$. By Proposition \ref{lift}, we have that $T\sqcup \bigsqcup_{i=1}^t \{\sigma(S_i)\}$ is a UFIM. 

Now (\ref{phiunique1}) and (\ref{phiunique2}) follow from the definitions of $K_1(G)$ and $N_1(G)$. Since $k(T) \le k(T) + \sum_{i=1}^t k(\{\sigma(S_i)\}) = k(T\sqcup \bigsqcup_{i=1}^t \{\sigma(S_i)\}) \le K_1(\ker(\phi))$, (\ref{split3}) follows.  By (\ref{split3}) and Remark \ref{phibound}, we have $k(S) \le K_1(\ker(\phi)) + k(S') \le K_1(\ker(\phi)) + k(\phi(S'))$, and this is (\ref{phibound2,4}). 

Since $\phi(S'')$ is a UFIM over $(G/\text{ker}(\phi))\setminus\{0\}$, we have
$$k(\phi(S')) = k(\phi(S'')) + \sum_{i=1}^t k(\phi(S_i)) \le K_1(G/\ker(\phi)) + \sum_{i=1}^t k(\phi(S_i)).$$
By the maximality of $t$, $\phi(S_i)$ must be a minimal zero-sum indexed multiset over $(G/\ker(\phi))\setminus\{0\}$ which implies $k(\phi(S_i)) \le K(G/\text{ker}(\phi))$ and thus $\sum_{i=1}^t k(\phi(S_i)) \le t\cdot K(G/\ker(\phi))$. Putting this all together, we have $k(\phi(S')) \le K_1(G/\ker(\phi)) + t\cdot K(G/\ker(\phi))$, which is (\ref{bound5}).

Now from (\ref{phibound2}), we have $t\le N_1(\text{ker}(\phi)) - |T|$. Furthermore, observe that $k(T) \le |T|$ and $K(G) \ge 1$ for any finite abelian group $G$. Hence by Remark \ref{phibound}, (\ref{split3}), and (\ref{bound5}), we have 
\begin{align*}
k(S) &= k(T) + k(S') \le k(T) + k(\phi(S'))\le k(T) + K_1(G/\ker(\phi)) + t\cdot K(G/\ker(\phi)) 
\\ &\le k(T) + K_1(G/\ker(\phi)) + (N_1(\ker(\phi)) - |T|)\cdot K(G/\ker(\phi))
\\ &\le k(T) - |T| + K_1(G/\ker(\phi)) + N_1(\ker(\phi))\cdot K(G/\ker(\phi)) 
\\ &\le  K_1(G/\ker(\phi)) + N_1(\ker(\phi))\cdot K(G/\ker(\phi)) \end{align*}
and (\ref{generalbound6}) follows. Finally, (\ref{push7}) follows from (\ref{phibound2,4}) and taking $t=0$ in (\ref{bound5}).
\end{proof}

\section{First Main Result}
\label{mainresult}
In this section we prove our first main results, namely Theorem \ref{mainthm} and Corollary \ref{corollary}. We first make the following remark.
\begin{remark} \label{N1K1} For any prime $p$ and $n\in \mathbb{N}$, we have 
$$N_1(C_p^n) = p K_1(C_p^n),$$
since each nonzero element of $C_p^n$ has order $p$.
\end{remark}

\begin{proof}[Proof of Theorem \ref{mainthm}]
We prove (\ref{1}) and (\ref{2}) separately.
\\

\noindent \textbf{Proof of (\ref{1}):} 
Suppose $S$ is a UFIM over $(C_{p^m}\oplus C_p^n)\setminus\{0\}$. Put $T_k = \{x\in S : \ord(x) = p^k\}$, and put $a_k = |T_k|$. Let $\phi: C_{p^m}\oplus C_p^n \rightarrow C_{p^{m-1}}$ be the ``multiplication by $p$" map. Now in the notation of Construction \ref{construction}, we have that $T = T_1$ so that $S' = S\setminus T_1$.

We have 
$$k(S) = \frac{a_1}{p} + \frac{a_2}{p^2} + \cdots + \frac{a_m}{p^m}\hspace{2mm} \text{which implies}\hspace{2mm} k(S') = \frac{a_2}{p^2} + \frac{a_3}{p^3}+ \cdots+ \frac{a_m}{p^m},$$
and since $\phi$ is the ``multiplication by $p$" map, 
$$k(\phi(S')) = p\cdot k(S') = \frac{a_2}{p} + \frac{a_3}{p^2}+ \cdots + \frac{a_m}{p^{m-1}}.$$ 
So now we have
$$k(S) = k(S') + k(T_1)= \frac{1}{p}\cdot k(\phi(S')) + k(T_1)$$
and by Proposition \ref{bound} (\ref{bound5}), we have 
$$k(\phi(S')) \le K_1(G/\ker(\phi)) + t\cdot K(G/\ker(\phi)) = K_1(C_{p^{m-1}}) + t\cdot K(C_{p^{m-1}}).$$
so we have
$$k(S) \le \frac{1}{p}\cdot\left[K_1(C_{p^{m-1}}) + t\cdot K(C_{p^{m-1}})\right] + k(T_1).$$
Note that $T_1 = \{x\in S: x\in \ker(\phi)\}$, $|T_1| = a_1$ and by Proposition \ref{N1} and Corollary \ref{remark} $(2)$ and Remark \ref{N1K1} we have 
$$a_1 + t = |T_1| + t \le N_1\left(\text{ker}(\phi)\right) = N_1(C_p^{n+1}) = K_1(C_p^{n+1})p\hspace{2mm} \text{which implies} \hspace{2mm} t\le K_1(C_p^{n+1})p-a_1.$$
Note also that $k(T_1) = \frac{a_1}{p}.$ By Proposition \ref{Kbound}, we have $K(C_{p^{m-1}})= 1$, and by Theorem \ref{GaoWang} we have $K_1(C_{p^{m-1}}) = 1+\frac{1}{p}+\cdots+\frac{1}{p^{m-1}}$, so in all we have 
\begin{align*}
k(S) &\le \frac{1}{p}\cdot\left[K_1(C_{p^{m-1}}) + t\cdot K(C_{p^{m-1}})\right] + \frac{a_1}{p} \le \frac{1}{p}\cdot\left[1+ \frac{1}{p}+ \cdots +\frac{1}{p^{m-2}} + \left(K_1(C_p^{n+1})p-a_1\right)\right]
\\& + \frac{a_1}{p} = K_1(C_{p^m})+ K_1(C_p^{n+1}) -1.\end{align*} \qed

\noindent \textbf{Proof of (\ref{2}):} Suppose $S$ is a UFIM over $(C_{p^m}\oplus C_q^n)\setminus\{0\}$. Let $T_{ij} = \{x\in S : \text{ord}(x) = p^iq^j\}$ for $0\le i\le m$ and $0 \le j \le 1$ (note $T_{00} = \emptyset$), and put $a_{ij} = |T_{ij}|$. Then we have 
$$k(S) = \frac{a_{10}}{p} + \cdots + \frac{a_{m0}}{p^m} + \frac{a_{11}}{pq} + \cdots +\frac{a_{m1}}{p^mq} + \frac{a_{01}}{q}.$$
Let $\phi_1 : C_{p^m} \oplus C_q^n \rightarrow C_{p^m}$ and $\phi_2 : C_{p^m} \oplus C_q^n \rightarrow C_q^n$ be the canonical projections. Now in the notation of Construction \ref{construction}, we have that $T(\phi_1) = T_{01}$, so that $S' = S\setminus T_{01}$. Then we have
$$k(\phi_1(S')) = \frac{a_{10}}{p} + \cdots + \frac{a_{m0}}{p^m} + \frac{a_{11}}{p} + \cdots +\frac{a_{m1}}{p^m}.$$ 
Now note that $T_{01} = \{x\in S: x\in \text{ker}(\phi_1)\}$ and $|T_{01}| + t = a_{01} + t$, so by Corollary \ref{remark} (2) and Remark \ref{N1K1} we have that 
$$t\le N_1(C_q^n)-a_{01} = qK_1(C_q^n) - a_{01}.$$
From Proposition \ref{Kbound}, $K(C_{p^m}) = 1$, and by Theorem \ref{GaoWang}, we have $K_1(C_{p^m}) = 1 + \frac{1}{p} + \cdots +\frac{1}{p^{m-1}}$. So now by Proposition \ref{bound} (\ref{bound5}), 

\begin{align*}
k(\phi_1(S')) &\le K_1(G/\text{ker}(\phi_1)) + t\cdot K(G/\text{ker}(\phi_1))= K_1(C_{p^m}) + t\cdot K_1(C_{p^m}) 
\\ &\le 1 + \frac{1}{p} + \cdots + \frac{1}{p^{m-1}} + (qK_1(C_q^n)-a_{01}).
\end{align*}
Thus
\begin{align*}
k(S') &= \frac{a_{10}}{p} + \cdots + \frac{a_{m0}}{p^m} + \frac{a_{11}}{pq} + \cdots + \frac{a_{m1}}{p^mq}
\\ &= \frac{1}{q}\cdot k(\phi_1(S')) + \frac{q-1}{q}\left[\frac{a_{10}}{p} + \cdots + \frac{a_{m0}}{p^m}\right] 
\\ &\le \frac{1}{q}\left[1 + \frac{1}{p} + \cdots + \frac{1}{p^{m-1}} + (qK_1(C_q^n)-a_{01})\right] + \frac{q-1}{q}\left[\frac{a_{10}}{p} + \cdots + \frac{a_{m0}}{p^m}\right]. \end{align*} 
Now in the notation of Construction \ref{construction} with respect to the homomorphism $\phi_2$, let $T(\phi_2) = \{x\in S: x\in \ker(\phi_2)\}$. Notice that $T(\phi_2) = \bigsqcup_{i=1}^{m_0}T_{i0}$, so by Corollary \ref{remark} (1), we know that $k(T(\phi_2)) \le K_1(C_{p^m})$, i.e.
$$\frac{a_{10}}{p} + \cdots + \frac{a_{m0}}{p^m} - \left(1 + \frac{1}{p} + \cdots + \frac{1}{p^{m-1}}\right)\le 0.$$ 
So now by Remark \ref{phibound}, we have 
\begin{align*}
k(S) &= k(S') + k(T_{01}) \le \frac{1}{q}\left[1 + \frac{1}{p} + \cdots + \frac{1}{p^{m-1}} + (qK_1(C_q^n)-a_{01})\right] 
\\ &+  \frac{q-1}{q}\left[\frac{a_{10}}{p} + \cdots + \frac{a_{m0}}{p^m}\right] + \frac{a_{01}}{q} = 1 + \frac{1}{p} + \cdots + \frac{1}{p^{m-1}} + K_1(C_q^n) 
\\ &+ \frac{q-1}{q}\left[\frac{a_{10}}{p} + \cdots \frac{a_{m0}}{p^m} - \left(1+\frac{1}{p} + \cdots + \frac{1}{p^{m-1}}\right)\right] 
\le K_1(C_{p^m}) + K_1(C_q^n). \end{align*}
\end{proof}

\begin{remark} \label{maximalstructurepq} Note that when $m=n=1$ in the proof of (2), i.e. when $G = C_{pq}$, we may deduce more about the structure of a UFIM achieving \emph{maximal} cross number as follows: From the last chain of inequalities in the proof of (2), we have 
$$k(S) \le K_1(C_p) + K_1(C_q) + \frac{q-1}{q}\left(\frac{a_{10}}{p} -1\right)\le K_1(C_p) + K_1(C_q) = 2 = K_1^*(C_{pq})$$
with equality holding only if $\frac{a_{10}}{p} = 1$. By symmetry, we also have $\frac{a_{01}}{q} = 1$. Hence, for a maximal cross number-achieving UFIM $S$, 
$$2 = \frac{a_{10}}{p} + \frac{a_{01}}{q} \le \frac{a_{10}}{p} + \frac{a_{01}}{q} + \frac{a_{11}}{pq}\le k(S) \le K_1(C_{pq}) = 2 \hspace{2mm} \text{which implies} \hspace{2mm} a_{11} = 0,$$
that is, $S$ has no ``cross terms" in the sense of Remark \ref{crossterms}, and so we may split $S$ into a disjoint union $S_p\sqcup S_q$ where $S_p$ is a UFIM over $C_p\setminus\{0\}$ and $S_q$ is a UFIM over $C_q\setminus\{0\}$. 
\end{remark}
\begin{proof}[Proof of Corollary \ref{corollary}]
This follows directly from Proposition \ref{lowerbound}, Theorem \ref{GaoWang} and Theorem \ref{mainthm}.

\end{proof}

\section{Structural Results}\label{structuralresults}
We now prove some results which give us information on the structure of UFIMs in relation to the structure of the ambient group. In particular, Lemma \ref{lowestorder} will comprise a key step in proving our second main result in Section \ref{2pq} by allowing us to derive a stronger upper bound for the cross number when there are ``few" elements of lowest possible order. 

For any $n\in \mathbb{Z}$, let $P^-(n)$ denote the smallest (positive) prime divisor of $n$, and let $P^+(n)$ denote the largest prime divisor of $n$. 
\begin{proposition}[\cite{NarkiewiczSliwa}] \label{factorproduct} \label{mbound} \label{ksum} Let $G$ be a finite abelian group and let $S$ be a UFIM over $G\setminus\{0\}$. Then if $\bigsqcup_{i=1}^m I_i$ is the irreducible factorization of $S$, we have 
$$\prod_{i=1}^m |I_i| \le |G|.$$
Furthermore we have $m \le \log_2 |G|$, and for any choice of $g_i\in S(I_i)$ for each $1\le i\le m$, we have
$$\sum_{i=1}^m k(\{g_i\}) \le \frac{m}{P^-(|G|)} \le \frac{\log_2 |G|}{P^-(|G|)}.$$
\end{proposition}

\begin{proof} For the first statement, see \cite{NarkiewiczSliwa}. For each irreducible factor $S(I_i)$, since $S(I_i)$ is zero-sum over $G\setminus\{0\}$, we have $|I_i| = |S(I_i)| \ge 2$, and so $2^m \le \prod_{i=1}^m |I_i| \le |G|$ which implies $m\le \log_2 |G|$. Now since  $\ord(g) \ge P^-(|G|)$ for all $g\in G\setminus\{0\}$, the third statement follows.

\end{proof}

We now prove a statement that gives a lower bound for the number of irreducible factors for a hypothetical counterexample to Conjecture \ref{K1conjecture}. 

\begin{proposition} \label{sizelimit} \label{plengthbound} Let $G = \bigoplus_{i=1}^n \bigoplus_{j=1}^{n_i} C_{p_i^{e_{ij}}}$ be an abelian group (written as a direct sum of prime-power cyclic groups) such that $k(G) = k^*(G)$. Let $S$ be a UFIM over $G\setminus\{0\}$ with irreducible factorization $\bigsqcup_{i=1}^m I_i$. Then if 
$$m \le \sum_{i=1}^n \sum_{j=1}^{n_i} \frac{P^-(|G|)}{p_i} K_1^*(C_{p_i^{e_{ij}}})$$
then $k(S) \le K_1^*(G)$. In particular, for a $p$-group $G$, if $m\le K_1^*(G)$, then $k(S) \le K_1^*(G)$.
\end{proposition}

\begin{proof} If $n>1$, assume without loss of generality that $P^-(|G|) = p_1 < \cdots < p_n$. For each $1\le i \le m$ choose a $g_i\in S(I_i)$, and observe that by unique factorization, both $\bigsqcup_{i=1}^m S(I_i)\setminus\{g_i\}$ and $\bigsqcup_{i=1}^m \{g_i\}$ are zero-sum free, so that $k\left(\bigsqcup_{i=1}^m S(I_i)\setminus\{g_i\}\right) \le k(G)$. By Proposition \ref{ksum}, we have $\sum_{i=1}^m k(\{g_i\}) \le \frac{m}{p_1}$. Now note that by our assumption on $m$,
$$K_1^*(G) - k^*(G) = \sum_{i=1}^n \sum_{j=1}^{n_i} \frac{p_i^{e_{ij}}-1}{p_i^{e_{ij}} - p_i^{e_{ij}-1}} - \sum_{i=1}^n \sum_{j=1}^{n_i} \frac{p_i^{e_{ij}}-1}{p_i^{e_{ij}}} = \sum_{i=1}^n \sum_{j=1}^{n_i} \frac{1}{p_i}K_1^*(C_{p_i^{e_{ij}}})\ge \frac{m}{p_1},$$
and hence  $\sum_{i=1}^m k(\{g_i\}) \le \frac{m}{p_1} \le K_1^*(G) - k^*(G) = K_1^*(G) - k(G)$. Thus
$$k(S) = k\left(\bigsqcup_{i=1}^m S(I_i)\setminus\{g_i\}\right) + \sum_{i=1}^m k(\{g_i\}) \le k(G) + K_1^*(G) - k(G) = K_1^*(G).$$
The statement for $p$-groups follows immediately by taking $n=1$.
\end{proof} 

Intuitively, it would seem that in order for the cross number of a indexed multiset $S$ to be large, low-order elements should be in some sense ``common" in $S$. The following lemma studies the effect on $k(S)$ of the distribution of elements of lowest possible order among the irreducible factors of a UFIM $S$. In particular, if none of the irreducible factors of $S$ consist entirely of elements of lowest possible order, then for certain classes of finite abelian groups we shall be able to prove that $k(S)$ will ``eventually" be less than $K_1^*(G)$ (see Corollary \ref{noporder}).  

\begin{lemma} \label{lowestorder} Suppose $G = \bigoplus_{i=1}^n\bigoplus_{j=1}^{n_i} C_{p_i^{e_{ij}}}$ is a finite abelian group with $p_1 < \cdots < p_n$ and which does not satisfy both $n=1$ and $\max_{1\le j \le n_1} e_{1j} =1$ (i.e. $G$ is not an elementary $p_1$-group). For any UFIM $S$ over $G\setminus\{0\}$, let $S_{p_1}$ be the union of all irreducible factors of $S$ whose elements are contained in $C_{p_1}^{n_1}$, so that $S_{p_1}$ is a UFIM over $C_{p_1}^{n_1}\setminus\{0\}$ (note that possibly $S_{p_1} = \emptyset$). Let $m_{p_1}$ be the number of irreducible factors of $S_{p_1}$. Then
\[k(S)  \le 
\begin{cases} 
      k(G) + \frac{(\log_2 |G|)-m_{p_1}}{\min\{p_1^2, p_2\}} + \frac{m_{p_1}}{p_1} &  n>1 \; \emph{and} \;\max_{1\le j\le n_1} e_{1j} > 1 \\
      k(G) + \frac{(\log_2 |G|)-m_{p_1}}{p_2} + \frac{m_{p_1}}{p_1} & n>1 \; \emph{and} \; \max_{1\le j \le n_1} e_{1j} = 1\\
      k(G) + \frac{(\log_2 |G|)-m_{p_1}}{p_1^2} + \frac{m_{p_1}}{p_1} & n=1\; \emph{and}\; \max_{1\le j\le n_1} e_{1j} > 1.
\end{cases}
\]
\end{lemma}

\begin{proof} Let $\bigsqcup_{i=1}^m I_i$ be the irreducible factorization of $S$. For each $1\le i\le m$, choose some $g_i\in S(I_i)$; for each $i$ such that $S(I_i) \not \subset S_{p_1}$, we can choose $g_i$ such that $\ord(g_i) > p_1$,  and note that $\sum_{g_i\in S_{p_1}} k(\{g_i\}) = \frac{m_{p_1}}{p_1}$. Now
\begin{enumerate}
\item if $n>1$ and $\max_{1\le j\le n_1} e_{1j} >1$, we have $\ord(g_i) \ge \min\{p_1^2,p_2\}$,
\item if $n>1$ and $\max_{1\le j\le n_1} e_{1j} =1$, we have $\ord(g_i) \ge p_2$, and
\item if $n=1$ and $\max_{1\le j\le n_1} e_{1j} >1$, we have $\ord(g_i) \ge p_1^2$. 
\end{enumerate}
So by Proposition \ref{ksum}, since $S\setminus S_{p_1}$ has $m-m_{p_1}$ irreducible factors, we have
\[\sum_{g_i\not\in S_{p_1}} k(\{g_i\}) \le 
\begin{cases}
      \frac{(\log_2 |G|)-m_{p_1}}{\min\{p_1^2,p_2\}} & n>1 \; \text{and} \; \max_{1\le j\le n_1} e_{1j} >1\\
      \frac{(\log_2 |G|)-m_{p_1}}{p_2} & n>1 \; \text{and} \; \max_{1\le j \le n_1} e_{1j} = 1\\
      \frac{(\log_2 |G|)-m_{p_1}}{p_1^2} &  n=1\; \text{and}\; \max_{1\le j\le n_1} e_{1j} > 1.
\end{cases}\]

By unique factorization, $\bigsqcup_{i=1}^m S(I_i)\setminus\{g_i\}$ is zero-sum free, and so $k(\bigsqcup_{i=1}^m S(I_i)\setminus\{g_i\})\le k(G)$. This, combined with the above, gives the desired conclusion.
\end{proof}

For a given $c\in \mathbb{R}_{\ge 1}$, we define the following subset of the set of finite abelian groups $G$: 
$$\Omega_c : = \{G : P^+(|G|) \le c\cdot P^-(|G|)\}.$$
For a given finite abelian group $G = \bigoplus_{i=1}^n\bigoplus_{j=1}^{n_i} C_{p_i^{e_{ij}}}$, define
$$\mathcal{S}_N := \left\{G = \bigoplus_{i=1}^n \bigoplus_{j=1}^{n_i} C_{p_i^{e_{ij}}} : \sum_{i=1}^n\sum_{j=1}^{n_i} e_{ij} \le N\right\}.$$
Note that $\mathcal{S}_N$ consists of those finite abelian groups $G$ such that the number of prime divisors of $|G|$ counted with multiplicity is at most $N$.

\begin{proposition} \label{pbound} Suppose $c,N\in\mathbb{R}_{\ge 1}$ and $G = \bigoplus_{i=1}^n\bigoplus_{j=1}^{n_i} C_{p_i^{e_{ij}}}\in \Omega_c\cap \mathcal{S}_N$ with $k(G) = k^*(G), \;p_1< \cdots < p_n$ and $p_1^2<p_2$ if $n>1$, $\max_{1\le j\le n_1} e_{1j} >1$, and $p_1$ large enough so that $\frac{\log_2 cp_1}{p_1}N \le \frac{1}{c}K_1^*(G)$. Then given a UFIM $S$ over $G\setminus\{0\}$, let $m_{p_i}$ be as in the statement of Lemma \ref{lowestorder}, and we have
$$k(S) \le K_1^*(G) + \frac{m_{p_1}}{p_1}\left(1-\frac{1}{p_1}\right).$$
In particular, if $n=1$ (i.e., $G$ is a $p$-group), $k(G) = k^*(G)$ by Remark \ref{kremark}, and so taking $c=1$, $G$ satisfies the above inequality.
\end{proposition}

\begin{proof}
By Lemma \ref{lowestorder} and since $G\in \Omega_c\cap \mathcal{S}_N$, we have 
$$k(S) \le k(G) + \frac{\log_2 cp_1}{p_1^2}\sum_{i=1}^n\sum_{j=1}^{n_i} e_{ij} -\frac{m_{p_1}}{p_1^2} + \frac{m_{p_1}}{p_1}\le k(G) + \frac{\log_2 cp_1}{p_1^2}N + \frac{m_{p_1}}{p_1}\left(1-\frac{1}{p_1}\right).$$
Now since we assume $k(G) = k^*(G)$ and $G\in \mathcal{S}_N$,
$$K_1^*(G) - k(G) = K_1^*(G) - k^*(G) = \sum_{i=1}^n\sum_{j=1}^{n_i} \frac{1}{p_i}K_1^*(C_{p_i^{e_{ij}}}) \ge\sum_{i=1}^n\sum_{j=1}^{n_i} \frac{1}{cp_1}K_1^*(C_{p_i^{e_{ij}}}) = \frac{1}{cp_1}K_1^*(G).$$
So now for all $p_1$ large enough so that $\frac{\log_2 cp_1}{p_1}N \le \frac{1}{c}K_1^*(G)$, by the above we have
$$k(S) \le k(G) +  \frac{1}{cp_1}K_1^*(G) + \frac{m_{p_1}}{p_1}\left(1-\frac{1}{p_1}\right) \le K_1^*(G) + \frac{m_{p_1}}{p_1}\left(1-\frac{1}{p_1}\right).$$
\end{proof}

\begin{corollary} \label{pbehavior} For any $c,N\in\mathbb{R}_{\ge 1}$ and any finite abelian group $G = \bigoplus_{i=1}^n\bigoplus_{j=1}^{n_i} C_{p_i^{e_{ij}}}\in \Omega_c\cap \mathcal{S}_N$ with $k(G) = k^*(G),\; p_1< \cdots < p_n$ and $p_1^2<p_2$ if $n>1$, $\max_{1\le j\le n_1} e_{1j} >1$, and $p_1$ large enough so that $\frac{\log_2 cp_1}{p_1}N \le \frac{1}{c}K_1^*(G)$, we have
$$K_1(G) \le K_1^*(G) + n_1\frac{\log_2 p_1}{p_1}\left(1-\frac{1}{p_1}\right).$$
In particular, if $n=1$, $k(G) = k^*(G)$ by Remark \ref{kremark}, so $G$ satisfies this inequality with $c=1$.
\end{corollary}
\begin{proof} 
By Proposition \ref{mbound}, we have $m_{p_1} \le \log_2 |C_{p_1}^{n_1}| = n_1\log_2 p_1$, and so invoking Proposition \ref{pbound}, we have 
$$k(S) \le K_1^*(G) + \frac{m_{p_1}}{p_1}\left(1-\frac{1}{p_1}\right) \le K_1^*(G) + n_1\frac{\log_2 p_1}{p_1}\left(1-\frac{1}{p_1}\right)$$
for any UFIM $S$ over $G\setminus\{0\}$, and so the conclusion follows. 
\end{proof}

\begin{corollary}\label{noporder} Suppose $c,N\in\mathbb{R}_{\ge 1}$ and $G = \bigoplus_{i=1}^n\bigoplus_{j=1}^{n_i} C_{p_i^{e_{ij}}}\in \Omega_c\cap \mathcal{S}_N$ with $k(G) = k^*(G), \; p_1< \cdots < p_n$ and $p_1^2<p_2$ if $n>1$, $\max_{1\le j\le n_1} e_{1j} >1$, and $p_1$ large enough so that  $\frac{\log_2 cp_1}{p_1}N \le \frac{1}{c}K_1^*(G)$. Then any UFIM $S$ over $G\setminus\{0\}$ with irreducible factorization $\bigsqcup_{i=1}^m I_i$ such that for all $1\le i\le m$, $S(I_i)$ contains an element outside of $C_{p_1}^{n_1}$, satisfies
$$k(S) \le K_1^*(G).$$
In particular, if $n=1$, $k(G) = k^*(G)$ by Remark \ref{kremark}, so $G$ satisfies this inequality with $c=1$.
\end{corollary}
\begin{proof} The conditions imply $m_{p_1} = 0$; plug this into the inequality provided by Proposition \ref{pbound}.
\end{proof}

We may observe from Corollary \ref{plengthbound} and Corollary \ref{noporder} that to study Conjecture \ref{K1conjecture} for the classes of groups specified in Proposition \ref{pbound}, we essentially only need to look at UFIMs $S$ with strictly greater than $K_1^*(G)$ irreducible factors, and such that some irreducible factor contains only elements of order $P^-(|G|)$. Note also that by Proposition \ref{mbound} the number of irreducible factors which contain only elements of order $P^-(|G|)$ is bounded above by $\log_2 |C_{P^-(|G|)}^r| = r \log_2 P^-(|G|)$, where $r$ is the rank of $G$ as defined in Section \ref{introduction}. 

\section{Second Main Result}\label{2pq}
\label{pqr}
We can now prove Theorem \ref{mainthm2} and Corollary \ref{secondmaincorollary}. 

\begin{lemma} \label{roplus} Suppose $G$ is a finite abelian group with $K_1(G) = K_1^*(G)$ and $r\in\{2,3\}$ is such that $r \nmid |G|$. Given a UFIM $S$ over $G\setminus\{0\}$, let $m_r$ be as in Lemma \ref{lowestorder}. Then in the notation of Construction \ref{construction} with respect to the projection $\phi : C_r\oplus G \rightarrow G$, either
\begin{enumerate}
\item $k(S) \le K_1^*(C_r\oplus G)$ and $t=0$, or 
\item $m_r = 0$. 
\end{enumerate}
\end{lemma}

\begin{proof} Suppose $r = 2$. In the notation of Construction \ref{construction} with respect to $\phi : C_r \oplus G \rightarrow G$, by Proposition \ref{phibound2} (\ref{phiunique2}) and Theorem \ref{N1} we have $|T| + t \le N_1(\ker(\phi)) = N_1(C_2) = 2$. If $t=0$, then by Proposition \ref{push} (\ref{push7}) and Theorem \ref{GaoWang}, $k(S) \le K_1(C_2) + K_1(G) = K_1^*(C_2) + K_1^*(G) = K_1^*(C_2\oplus G)$, where the first equality follows from Theorem \ref{GaoWang} and our assumption $K_1(G) = K_1^*(G)$. Hence we may assume that $t\ge 1$, which implies, by the above, $|T| \le N_1(\ker(\phi)) - t = 2 - t \le 1$. But any irreducible factor has length at least 2, so $m_2 = 0$. 

Suppose $r = 3$. As above, we have $|T| + t\le N_1(\ker(\phi)) = N_1(C_3) = 3$. If $t=0$, by Proposition \ref{push} (\ref{push7}), we have $k(S) \le K_1(C_3) + K_1(G) = K_1^*(C_3) + K_1^*(G) = K_1^*(C_3\oplus G)$ where the first equality follows from Theorem \ref{GaoWang} and our assumption $K_1(G) = K_1^*(G)$. So we may assume that $t\ge 1$. Hence $1\le t \le 3$. If $t\ge 2$, we have by the above that $|T| \le 1$, and so since any irreducible factor has length at least 2, so $|T| \le 1$ which implies $m_2 = 0$. If $t=1$, then we have $|T|\le 2$. By Proposition \ref{phiunique}, $T\sqcup \{\sigma(S_1)\}$ is a UFIM and hence zero-sum over $C_3\setminus\{0\}$. Since the only zero-sum indexed multisets over $C_3\setminus\{0\}$  of length at most 3 are $\{1,2\}$, $\{1,1,1\}$ and $\{2,2,2\}$, we have $T = \{(1,0,0)\}, \{(2,0,0)\}, \{(1,0,0),(1,0,0)\}$ or $\{(2,0,0),(2,0,0)\}$. Thus $T$ is zero-sum free and so is properly contained in an irreducible factor. Since by definition $T$ contains all order-3 elements in $S$, we have that $m_3 = 0$. 
\end{proof}

\begin{proof}[Proof of Theorem \ref{mainthm2}]
Take any UFIM $S$ over $(C_r\oplus G)\setminus\{0\}$. By Lemma \ref{roplus}, if $m_r\neq 0$, we have $k(S) \ge K_1^*(C_r\oplus G)$. So now assume $m_r = 0$. By Lemma \ref{lowestorder}, we have 
$$k(S)\le  k(C_r\oplus G) + \frac{\log_2 |C_r\oplus G|}{p_1}\le  k(C_r\oplus G) + \frac{\log_2(rc^{(\sum_{i=2}^{n} \sum_{j=1}^{n_i} e_{ij})} p_1^{(\sum_{i=1}^n \sum_{j=1}^{n_i} e_{ij})})}{p_1}.$$
Now since by assumption
$$k(C_r\oplus G) = k^*(C_r \oplus G) = 1 - \frac{1}{r} + \sum_{i=1}^n n_i - \sum_{i=1}^n \sum_{j=1}^{n_i} \frac{1}{p_i^{e_{ij}}},$$
we have for $p_1$ satisfying the constraint given in our statement
\begin{align*}
k(S) &\le  1 - \frac{1}{r} + \sum_{i=1}^n n_i - \sum_{i=1}^n \sum_{j=1}^{n_i} \frac{1}{p_i^{e_{ij}}} + \frac{\log_2 (rc^{(\sum_{i=2}^{n} \sum_{j=1}^{n_i} e_{ij})} p_1^{(\sum_{i=1}^n \sum_{j=1}^{n_i} e_{ij})})}{p_1}
\\ & \le K_1^*(C_r\oplus G) - \frac{1}{r} + \frac{\log_2 (rc^{(\sum_{i=2}^{n} \sum_{j=1}^{n_i} e_{ij})} p_1^{(\sum_{i=1}^n \sum_{j=1}^{n_i} e_{ij})})}{p_1} 
\\ &- \left( \frac{1}{p_1}K_1^*\left(\bigoplus_{j=1}^{n_1} C_{p_1^{e_{1j}}}\right) + \sum_{i=2}^n \sum_{j=1}^{n_i} \frac{(cp_1)^{e_{ij}}-1}{(cp_1)^{e_{ij}+1} - (cp_1)^{e_{ij}}}\right)\le K_1^*(C_r\oplus G)\\
\end{align*}
with equality only if we have equality in the constraint for $p_1$. Hence $K_1(C_r\oplus G) \le K_1^*(C_r\oplus G)$, and so $K_1(C_r\oplus G) = K_1^*(C_r \oplus G)$ by Proposition \ref{lowerbound}.

Now consider the case where equality does not hold in the constraint for $p_1$ in our statement. The above argument shows that if $m_r = 0$, then $k(S) < K_1(C_r\oplus G)$. Hence if $S$ is such that $k(S) = K_1(C_r\oplus G)$, then $m_r\neq 0$, and so by Lemma \ref{roplus}, $t=0$. Thus $S$, by Remark \ref{crossterms}, each element of $S$ must belong to either $C_r$ or $C_{pq}$, and we may split $S$ into a disjoint union $S_r\sqcup S_G$ where $S_r$ is a UFIM over $C_r\setminus\{0\}$ and $S_G$ is a UFIM over $G\setminus\{0\}$.
\end{proof}

\begin{proof}[Proof of Corollary \ref{secondmaincorollary}] All the families described are covered in Proposition \ref{Kbound}, so by Remark \ref{kremark}, $k(C_r\oplus G) = k^*(C_r \oplus G)$ for the above $G$. Moreover, for the above $G$ we have $K_1(G) = K_1^*(G)$ by Theorem \ref{mainthm}. Hence the hypotheses of Theorem \ref{mainthm2} are satisfied for these $G$, and so the first part of the statement follows. The second part follows directly from Theorem \ref{mainthm2}.
\end{proof}

\begin{remark}
Note that for $r =2 ,3$ for $S$ over $C_{rpq}\setminus\{0\}$ for $p,q$ satisfying the conditions on $p$ and $q$ in the statement of Corollary \ref{secondmaincorollary}, any UFIM $S$ which achieves maximal cross number must have a decomposition $S_r\sqcup S_{pq}$ where $S_r$ is a UFIM over $C_r\setminus\{0\}$ and $S_{pq}$ is a UFIM over $C_{pq}\setminus\{0\}$. Hence $S$ achieves maximal cross number if and only if $S_r$ and $S_{pq}$ achieve maximal cross number. By Remark \ref{maximalstructurepq}, if $S_{pq}$ achieves maximal cross number, then it has a decomposition $S_p\sqcup S_q$, where $S_p$ is a UFIM over $C_p\setminus\{0\}$ and $S_q$ is a UFIM over $C_q\setminus\{0\}$, and so $S$ has a decomposition $S_r \sqcup S_p \sqcup S_q$.
\end{remark}

\section{Bounds on $K_1(G)$ and Asymptotic Results}
\label{secbound}
In this section, we prove some general bounds on $K_1(G)$. As a result, we show that $K_1(G),\; k(G)$ and $K_1^*(G)$ all become arbitrarily close to each other in a certain limit. We hope these results along with those of Section \ref{structuralresults} will be helpful in proving (or disproving) Conjecture \ref{K1conjecture} for further families of groups.

Gao and Wang give the following general bound for $K_1(G)$.
\begin{proposition}[\cite{GaoWang}] \label{GaoWangbound} For any finite abelian group $G$, let $|G|$ denote the order of $G$ and let $p$ be the smallest prime dividing $G$. Then we have
$$K_1(G) \le \log|G| +\frac{\log_2 |G|}{p}.$$
\end{proposition}
This bound can be improved however by refining Gao and Wang's methods in \cite{GaoWang}.
\begin{proposition}[\cite{Girardcorrespondence}] \label{Girardbound} For any finite abelian group $G$, we have 
$$K_1(G) \le 2k(G).$$
\end{proposition}

\begin{proof} For any given UFIM $S$, let $\bigsqcup_{i=1}^m I_i$ be its irreducible factorization. Then for each $1\le i\le m$, pick some $g_i \in S(I_i)$, and by unique factorization we have that $\bigsqcup_{i=1}^m S(I_i)\setminus\{g_i\}$ and $\bigsqcup_{i=1}^m \{g_i\}$ are zero-sum free. Hence, by the definition of $k(G)$,
$$k(S) = k\left(\bigsqcup_{i=1}^m S(I_i)\setminus\{g_i\}\right) + k\left(\bigsqcup_{i=1}^m \{g_i\}\right) \le 2k(G).$$
Since $S$ was an arbitrary UFIM, we have $K_1(G) \le 2k(G)$. 
\end{proof}

The following asymptotic result which more precisely captures the behavior of $K_1(G)$, in particular showing that it approaches the little cross number $k(G)$ in a certain limit. Recall the definitions of $\Omega_c$ and $\mathcal{S}_N$ as defined in Section \ref{structuralresults}.
\begin{proposition} \label{asymptote} For any $c,N \in \mathbb{R}_{\ge1}$, we have 
$$K_1(G) - k(G) \le N\frac{\log_2P^+(|G|)}{P^-(|G|)}$$
for all $G\in \mathcal{S}_N$ (note for $p$-groups, $P^+(|G|) = P^-(|G|)$). In particular, this implies that
$$\lim_{P^-(|G|)\rightarrow \infty,\; G\in \Omega_c\cap \mathcal{S}_N} |K_1(G) - k(G)| = 0.$$
\end{proposition}

\begin{proof} Write $G = \bigoplus_{i=1}^n \bigoplus_{j=1}^{t_i} C_{p_i^{e_{ij}}},\; p_1 < \cdots < p_n$. Let $S$ be any UFIM over $G\setminus\{0\}$ with irreducible factorization $\bigsqcup_{i=1}^m I_i$. Now for each $1\le i\le m$, choose any $g_i\in S(I_i)$, and by unique factorization $\bigsqcup_{i=1}^m S(I_i)\setminus\{g_i\}$ is zero-sum free, so $k(\bigsqcup_{i=1}^m S(I_i)\setminus\{g_i\}) \le k(G)$. Now
$$K(s) - k(G) \le \sum_{i=1}^m k(\{g_i\}) \le \frac{\log_2 |G|} {P^-(|G|)} = \frac{\sum_{i=1}^n\sum_{j=1}^{t_i} e_{ij}\log_2 p_i}{p_1}\le N \frac{\log_2 cp_1}{p_1}\rightarrow 0,$$
by Proposition \ref{factorproduct} and the assumption $G \in\Omega_c\cap\mathcal{S}_N$, as $p_1\rightarrow \infty$.
\end{proof}

For any $n\in \mathbb{N}_{\ge 1}$, let $\omega(n)$ denote the number of prime divisors of $n$ counted without multiplicity, and let
$$\mathcal{E}_{(l_1,\ldots,l_r)} := \left\{\bigoplus_{i=1}^r C_{n_i}, 1< n_1 | \cdots |n_r : \forall 1\le i\le r,\; \omega(n_i) = l_i, \; \gcd\left(n_i, \frac{n_r}{n_i}\right) = 1\right\}.$$

\begin{proposition}[\cite{Girard}] \label{Girardasymptote} For any $r,\; l_1, \ldots, l_r \in \mathbb{N}_{\ge 0}$, writing $G = \bigoplus_{i=1}^r C_{n_i}$ we have  
$$\lim_{P^-(n_r) \rightarrow \infty,\; G \in \mathcal{E}_{(l_1,\ldots,l_r)}} k\left(\bigoplus_{i=1}^r C_{n_i}\right) = \sum_{i=1}^r l_i = \lim_{P^-(n_r) \rightarrow \infty,\; G\in \mathcal{E}_{(l_1,\ldots,l_r)}} k^*\left(\bigoplus_{i=1}^r C_{n_i}\right).$$
\end{proposition}

\begin{lemma} \label{k*K_1*} We have
$$\lim_{P^-(|G|)\rightarrow \infty} |k^*(G) - K_1^*(G)| = 0.$$
\end{lemma}
\begin{proof} Write $G = \bigoplus_{i=1}^n\bigoplus_{j=1}^{n_i} C_{p_i^{e_{ij}}},\; p_1 < \cdots < p_n$. As $p_1\rightarrow 0$,
$$K_1^*(G) - k^*(G) = \sum_{i=1}^n\sum_{j=1}^{n_i} \frac{p_i^{e_{ij}}-1}{p_i^{e_{ij}} - p_i^{e_{ij}-1}}-\frac{p_i^{e_{ij}}-1}{p_i^{e_{ij}}} \le \sum_{i=1}^n\sum_{j=1}^{n_i} \frac{p_1^{e_{ij}}-1}{p_1^{e_{ij}} - p_1^{e_{ij}-1}}-\frac{p_1^{e_{ij}}-1}{p_1^{e_{ij}}}\rightarrow 0.$$
\end{proof}

\begin{corollary} \label{K1behavior} For any fixed $c\in\mathbb{R}_{\ge 1}$ and $N,\; r,\; l_1, \ldots, l_r \in\mathbb{N}$, we have
$$\lim_{P^-(|G|) \rightarrow \infty, \; G\in \Omega_c\cap \mathcal{S}_N \cap \mathcal{E}_{(l_1,\ldots,l_r)}} |K_1(G) - K_1^*(G)| = 0.$$
\end{corollary}
\begin{proof}This follows directly from Proposition \ref{asymptote}, Proposition \ref{Girardasymptote} and Lemma \ref{k*K_1*}. 
\end{proof}
\section{Conclusion}
A full resolution of Conjecture \ref{K1conjecture} still seems far away, though it is hopeful that it could be verified for larger classes of abelian groups. General $p$-groups seems to be the most amenable ``next step," as several of the results in Section \ref{structuralresults} seem to suggest. Of course, a resolution of the conjecture for general $p$-groups would be, by Remark \ref{N1K1}, at least as strong verifying Conjecture \ref{N1conjecture} for groups of the form $C_p^n$, and this has only recently been verified for $n=2$ by Gao, Li, and Peng (see \cite{GaoLiPeng}). Other directions of pursuit are to extend the asymptotic results of Section \ref{secbound}, and to study the structure of UFIMs which achieve maximal cross number. The results of Remark \ref{maximalstructurepq} and Theorem \ref{mainthm2} perhaps suggest the following conjecture. 

\begin{conjecture} Let $G$ be a finite abelian group such that $G = \bigoplus_{i=1}^n G_{p_i}$ where $p_1,\ldots,p_n$ are distinct primes and $G_{p_i}$ is the Sylow $p_i$-group of $G$ for each $1 \le i\le n$. If $S$ is a UFIM over $G\setminus\{0\}$ with $k(S) = K_1(G)$, then $S$ has a decomposition 
$$S = \bigsqcup_{i=1}^n S_{p_i}$$
where $S_{p_i}$ is a UFIM over $G_{p_i}\setminus\{0\}$ for each $1\le i\le n$. 
\end{conjecture}

\noindent\textbf{Acknowledgments.} This research was conducted while I was a participant at the University of Minnesota Duluth REU program, supported by NSF/DMS grant 1062709 and NSA grant H98230-11-1-0224. I would like to thank Joe Gallian for his encouragement, advice and enthusiasm in running the program. I would also like to thank the program advisors Adam Hesterberg, David Rolnick and Eric Riedl for their valuable suggestions to preliminary versions of this paper. Finally, I thank the program visitors Yasha Berchenko-Kogan, Nathan Kaplan, Brian Lawrence, Krishanu Sankar, and Jonathan Wang for helpful discussions.

\end{document}